\documentclass[11pt]{article}
\usepackage{amsfonts,amscd,amssymb, amsmath}
\newtheorem{definition}{Definition}[section]

\newtheorem{proposition}[definition]{Proposition}
\newtheorem{corollary}[definition]{Corollary}
\newtheorem{remark}[definition]{Remark}

\newtheorem{theorem}[definition]{Theorem}
\newtheorem{example}[definition]{Example}

\def\mb{\mathbb}
\newcommand{\nat}{\mbox{$\;\natural \;$}}

\def\rawo\lonra{\longrightarrow}

\def\ot{\otimes}

\newcommand{\selabel}[1]{\label{se:#1}}

\allowdisplaybreaks[4]

\newenvironment{proof}{{\it Proof.}}{\hfill $ \square $ \vskip 4mm}

\begin{document}
\title{L-R-smash products and 
L-R-twisted tensor products of algebras
\thanks{Research partially supported by the CNCSIS project 
 ''Hopf algebras, cyclic homology and monoidal categories'', 
contract nr. 560/2009, CNCSIS code $ID_{-}69$.}}
\author {M\u ad\u alin Ciungu\\ University of Bucharest, 
Faculty of Mathematics and Informatics\\
Str. Academiei 14, RO-010014 Bucharest 1, Romania\\
\and
Florin Panaite\\
Institute of Mathematics of the 
Romanian Academy\\ 
PO-Box 1-764, RO-014700 Bucharest, Romania\\
e-mail: Florin.Panaite@imar.ro
}
\date{}
\maketitle

\begin{abstract}
We introduce a common generalization of the L-R-smash product and 
twisted tensor product of algebras, under the name 
L-R-twisted tensor product of algebras. We investigate some 
properties of this new construction, for instance we prove a result 
of the type ''invariance under twisting'' and we show that 
under certain circumstances L-R-twisted tensor products 
of algebras may be iterated. 
\end{abstract}
\section*{Introduction}
${\;\;\;}$
The L-R-smash product over a cocommutative Hopf algebra was introduced 
and studied in a series of papers  
\cite{b1}, \cite{b2}, \cite{b3}, \cite{b4}, with motivation and 
examples coming from the theory of deformation quantization. This 
construction was generalized in \cite{pvo} to the case of arbitrary 
bialgebras (even quasi-bialgebras), as follows:  
if $H$ is a bialgebra, ${\cal A}$ an $H$-bimodule algebra 
and ${\mb A}$ an  $H$-bicomodule algebra, the 
L-R-smash product   ${\cal A}\nat {\mb A}$  is an associative algebra 
structure defined on  ${\cal A}\ot {\mb A}$ by the 
multiplication rule 
$(\varphi \nat u)(\varphi '\nat u')=(\varphi \cdot u'_{<1>})(u_{[-1]}
\cdot \varphi ')\nat u_{[0]}u'_{<0>}$,  
for all $\varphi , \varphi '\in {\cal A}$ and $u, u'\in {\mb A}$. 
The usual smash product $A\# H$ (where $A$ is a left $H$-module algebra) 
is a particular case, namely $A\# H=A\nat H$ if we regard $A$ 
as an $H$-bimodule algebra with trivial right $H$-action (and 
$H$ with its canonical $H$-bicomodule algebra structure). 

On the other hand, if $A$, $B$ are associative algebras and 
$R:B\ot A\rightarrow A\ot B$ is a linear map satisfying certain axioms, 
then $A\ot B$ becomes an associative algebra with a 
multiplication defined in terms of $R$ and the multiplications of 
$A$ and $B$. This construction appeared in a number of contexts and 
under different names. Following \cite{Cap}, we call such an $R$ a 
{\em twisting map} and the algebra structure on $A\ot B$ afforded by it 
the {\em twisted tensor product} of $A$ and $B$ and denote it 
by $A\ot _RB$. 
The twisted tensor product of 
algebras may be regarded as a representative for the cartesian product 
of noncommutative spaces, better suited than the ordinary tensor product, 
see \cite{Cap}, \cite{jlpvo}, \cite{lpvo} for a detailed discussion and 
references. There are many examples of twisted tensor products 
of algebras (see for instance \cite{cibils}, \cite{stefan} for some 
concrete examples and classification results), the most relevant for us here being the usual 
smash product $A\# H$ (where $H$ is a bialgebra and $A$ is a 
left $H$-module algebra).   

If we look at a general L-R-smash product  ${\cal A}\nat {\mb A}$,  
we realize quickly that it is {\em not} an example of a 
twisted tensor product of algebras. It appears thus natural 
to try to find a more general construction able to include the 
L-R-smash product as a particular case; that is, 
to find a common generalization for the L-R-smash product 
and the twisted tensor product of algebras. We introduce such a 
construction in this paper, under the name 
{\em L-R-twisted tensor product of algebras}, denoted 
by $A\;_Q\ot _RB$, where $A$, $B$ are algebras and 
$R:B\ot A\rightarrow A\ot B$, 
$Q:A\ot B\rightarrow A\ot B$ are linear maps 
satisfying certain compatibility conditions. If 
$Q=id_{A\ot B}$, then  $A\;_Q\ot _RB$ coincides 
with the twisted tensor product $A\ot _RB$. 

We find then a number of properties of this new construction, 
inspired by properties of the L-R-smash product or by properties of the 
twisted tensor product of algebras (or both). For instance, 
we prove that an L-R-twisted tensor product  $A\;_Q\ot _RB$ 
with bijective $Q$ is isomorphic to a certain twisted tensor 
product  $A\ot _PB$, generalizing a result from \cite{pvo} 
stating that an L-R-smash product over a Hopf algebra 
with bijective antipode is isomorphic to a so-called 
diagonal crossed product. Also, by 
generalizing the corresponding result for twisted tensor products 
proved in \cite{jlpvo}, we show that under certain circumstances 
L-R-twisted tensor products may be iterated (but for achieving this, 
we prove first a result of this type for L-R-smash products that will 
serve as a guiding example). Finally, we prove a result of the type 
''invariance under twisting'' for L-R-twisted tensor products of algebras, 
as a common generalization of the corresponding one for 
twisted tensor products proved in \cite{jlpvo} and of an 
invariance under twisting for L-R-smash products that we prove here 
and use also as a guiding example.    
\section{Preliminaries}\selabel{1}
${\;\;\;}$
In this section we recall some definitions and results and fix  
notation that will be used throughout the paper. 
We work over a commutative field $k$. All algebras, linear spaces
etc. will be over $k$; unadorned $\ot $ means $\ot_k$. By ''algebra'' we 
always mean an associative unital algebra.

We recall from \cite{Cap}, \cite{VanDaele} that, given two algebras $A$, $B$ 
and a $k$-linear map $R:B\ot A\rightarrow A\ot B$, with notation 
$R(b\ot a)=a_R\ot b_R$, for $a\in A$, $b\in B$, satisfying the conditions 
\begin{eqnarray*}
&&a_R\otimes 1_R=a\otimes 1, \;\;\;\;1_R\otimes b_R=1\otimes b,  \\
&&(aa')_R\otimes b_R=a_Ra'_r\otimes b_{R_r}, \\
&&a_R\otimes (bb')_R=a_{R_r}\otimes b_rb'_R, 
\end{eqnarray*}
for all $a, a'\in A$ and $b, b'\in B$ (where $r$ is another 
copy of $R$), if we define on $A\ot B$ a new multiplication, by 
$(a\ot b)(a'\ot b')=aa'_R\ot b_Rb'$, then this new multiplication is associative 
and has unit $1\ot 1$. In this case, the map $R$ is called 
a {\bf twisting map} between $A$ and $B$ and the new algebra 
structure on $A\ot B$ is denoted by $A\ot _RB$ and called the 
{\bf twisted tensor product} of $A$ and $B$ afforded by $R$. 

Let $H$ be a bialgebra, ${\cal A}$ an $H$-bimodule algebra 
(with $H$-module structures denoted by $h\ot \varphi \mapsto 
h\cdot \varphi $ and $\varphi \ot h\mapsto \varphi \cdot h$ for all 
$h\in H$, $\varphi \in {\cal A}$) and ${\mb A}$ an 
$H$-bicomodule algebra (with $H$-comodule structures 
denoted by ${\mb A}\rightarrow H\ot {\mb A}$, $u\mapsto 
u_{[-1]}\ot u_{[0]}$ and ${\mb A}\rightarrow {\mb A}\ot H$, 
$u\mapsto u_{<0>}\ot u_{<1>}$ for all $u\in {\mb A}$). Define 
on ${\cal A}\ot {\mb A}$ the product 
$(\varphi \nat u)(\varphi '\nat u')=(\varphi \cdot u'_{<1>})(u_{[-1]}
\cdot \varphi ')\nat u_{[0]}u'_{<0>}$,  
for all $\varphi , \varphi '\in {\cal A}$ and $u, u'\in {\mb A}$. Then, by 
\cite{pvo}, this product defines on  ${\cal A}\ot {\mb A}$ a 
structure of associative algebra with unit $1_{{\cal A}}\ot 1_{{\mb A}}$, 
denoted by  ${\cal A}\nat {\mb A}$ and called the {\bf L-R-smash product} 
of ${\cal A}$ and ${\mb A}$. In particular, for ${\mb A}=H$, 
the multiplication of  ${\cal A}\nat H$ is defined by 
\begin{eqnarray*}
&&(\varphi \nat h)(\varphi '\nat h')=(\varphi \cdot h'_2)
(h_1 \cdot \varphi ')\nat h_2h'_1, \;\;\;\forall \;\varphi , \varphi '\in {\cal A}, \;h, h'\in H. 
\end{eqnarray*} 
\section{The definition of the L-R-twisted tensor product of algebras}
\setcounter{equation}{0}
\begin{proposition}\label{LRtwpr}
Let $A$ and $B$ be two (associative unital) algebras, and 
$R:B\otimes A\rightarrow A\otimes B$, 
$Q:A\otimes B\rightarrow A\otimes B$ two linear maps, 
with notation $R(b\otimes a)=a_R\otimes b_R$, 
$Q(a\otimes b)=a_Q\otimes b_Q$, for all $a\in A$, $b\in B$, 
satisfying the following conditions:
\begin{eqnarray}
&&a_R\otimes 1_R=a\otimes 1, \;\;\;\;1_R\otimes b_R=1\otimes b,  
\label{tw0} \\
&&(aa')_R\otimes b_R=a_Ra'_r\otimes b_{R_r}, \label{tw4} \\
&&a_R\otimes (bb')_R=a_{R_r}\otimes b_rb'_R, \label{tw5} \\
&&a_Q\otimes 1_Q=a\otimes 1, \;\;\;\; 1_Q\otimes b_Q=
1\otimes b, \label{tw0'} \\
&&(aa')_Q\otimes b_Q=a_qa'_Q\otimes b_{Q_q}, \label{tw4'} \\
&&a_Q\otimes (bb')_Q=a_{Q_q}\otimes b_Qb'_q, \label{tw5'} \\
&&b_R\otimes a_{R_Q}\otimes b'_Q=b_R\otimes a_{Q_R}\otimes b'_Q, 
\label{comb1} \\
&&a_R\otimes b_{R_Q}\otimes a'_Q=a_R\otimes b_{Q_R}\otimes a'_Q, 
\label{comb2}
\end{eqnarray}
for all $a, a'\in A$ and $b, b'\in B$, where $r$ (respectively $q$) 
is another copy of $R$ (respectively $Q$). 
If we define on $A\otimes B$ a multiplication by 
$(a\otimes b)(a'\otimes b')=a_Qa'_R\otimes b_Rb'_Q$,  
then this multiplication is associative and $1\otimes 1$ is the unit. 
This algebra structure will be denoted by 
$A\; _Q\otimes _RB$ and will be called the 
{\bf  L-R-twisted tensor product} of 
$A$ and $B$ afforded by the maps $R$ and $Q$.
\end{proposition}
\begin{proof}
The fact that $1\otimes 1$ is the unit is obvious using (\ref{tw0}) 
and (\ref{tw0'}), so we only prove the associativity. 
We compute (where $R=r={\cal R}$ and $Q=q=\tilde{Q}$): 
\begin{eqnarray*}
[(a\otimes b)(a'\otimes b')](a''\otimes b'')&=&(a_Qa'_R\otimes b_Rb'_Q)
(a''\otimes b'')\\
&=&(a_Qa'_R)_qa''_r\otimes (b_Rb'_Q)_rb''_q\\
&\overset{(\ref{tw4'})}{=}&a_{Q_q}a'_{R_{\tilde{Q}}}a''_r\otimes 
(b_Rb'_Q)_rb''_{\tilde{Q}_q}\\
&\overset{(\ref{tw5})}{=}&a_{Q_q}a'_{R_{\tilde{Q}}}a''_{{\cal R}_r}
\otimes b_{R_r}b'_{Q_{\cal R}}b''_{\tilde{Q}_q}\\
&\overset{(\ref{comb1}), (\ref{comb2})}{=}&a_{Q_q}a'_{\tilde{Q}_R}
a''_{{\cal R}_r}\otimes b_{R_r}
b'_{{\cal R}_Q}b''_{\tilde{Q}_q},
\end{eqnarray*}
\begin{eqnarray*}
(a\otimes b)[(a'\otimes b')(a''\otimes b'')]&=&(a\otimes b)
(a'_Qa''_R\otimes b'_Rb''_Q)\\
&=&a_q(a'_Qa''_R)_r\otimes b_r(b'_Rb''_Q)_q\\
&\overset{(\ref{tw4})}{=}&a_qa'_{Q_{{\cal R}}}a''_{R_r}\otimes 
b_{{\cal R}_r}(b'_Rb''_Q)_q\\
 &\overset{(\ref{tw5'})}{=}&a_{\tilde{Q}_q}a'_{Q_{{\cal R}}}a''_{R_r}\otimes  
b_{{\cal R}_r}b'_{R_{\tilde{Q}}}b''_{Q_q},
\end{eqnarray*}
and we see that the two terms are equal. 
\end{proof}
\begin{remark}
If $Q=id_{A\otimes B} $, then $A\; _Q\otimes _RB$ is the 
ordinary twisted tensor product of algebras $A\otimes _RB$. 
\end{remark}
\begin{example}
If $H$ is a bialgebra, ${\cal A}$ is an $H$-bimodule algebra and 
${\mb A}$ is an $H$-bicomodule algebra, with notation as before,  
define the maps 
\begin{eqnarray*}
&&R:{\mb A}\otimes {\cal A}\rightarrow {\cal A}\otimes {\mb A}, 
\;\;\;\;R(u\otimes \varphi )=u_{[-1]}\cdot \varphi \otimes u_{[0]}, \\
&&Q: {\cal A}\otimes {\mb A}\rightarrow {\cal A}\otimes {\mb A}, 
\;\;\;\;Q(\varphi \otimes u)=\varphi \cdot u_{<1>}\otimes u_{<0>},
\end{eqnarray*}
for all $\varphi \in {\cal A}$ and $u\in {\mb A}$. 
Then one checks that the maps $R$ and $Q$ satisfy the axioms in 
Proposition \ref{LRtwpr} and the L-R-twisted tensor product  
${\cal A}\; _Q\otimes _R{\mb A}$ coincides with the L-R-smash 
product ${\cal A}\nat {\mb A}$. In particular, for ${\mb A}=H$, 
the above maps are given, 
for all $\varphi \in {\cal A}$, $h\in H$, by 
\begin{eqnarray*}
&&R:H\otimes {\cal A}\rightarrow {\cal A}\otimes H, 
\;\;\;\;R(h\otimes \varphi )=h_1\cdot \varphi \otimes h_2, \\
&&Q: {\cal A}\otimes H\rightarrow {\cal A}\otimes H, 
\;\;\;\;Q(\varphi \otimes h)=\varphi \cdot h_2\otimes h_1.
\end{eqnarray*}
\end{example}

A particular case of Proposition \ref{LRtwpr} is obtained if 
$R$ is the flip map $b\otimes a\mapsto a\otimes b$:
\begin{corollary}
Let $A$, $B$ be two algebras and $Q:A\otimes B\rightarrow 
A\otimes B$ a linear map satisfying the conditions 
(\ref{tw0'}), (\ref{tw4'}) and (\ref{tw5'}). 
Then the multiplication $(a\otimes b)(a'\otimes b')=a_Qa'\otimes bb'_Q$ 
defines an associative algebra structure on $A\otimes B$ 
with unit $1\otimes 1$, denoted by $A\; _Q\otimes B$. 
\end{corollary} 
\begin{remark}
Let $A$, $B$ be two algebras and $Q:A\otimes B\rightarrow 
A\otimes B$ a linear map, with notation $Q(a\ot b)=a_Q\ot b_Q$. 
Define the map $Q^{op}:B\ot A\rightarrow A\ot B$, 
$Q^{op}(b\ot a)=a_Q\ot b_Q$. Then one can easily check that 
$Q$ satisfies the conditions   
(\ref{tw0'}), (\ref{tw4'}) and (\ref{tw5'}) if and only if 
$Q^{op}$ is a twisting map between the opposite algebras 
$A^{op}$ and $B^{op}$ and in this case we have an algebra 
isomorphism  $A\; _Q\otimes B\equiv (A^{op}\ot _{Q^{op}}B^{op})^{op}$, 
given by the trivial identification. 
\end{remark}
\begin{remark}
If  $A\; _Q\otimes _RB$ is an L-R-twisted tensor product of algebras, 
we can consider also the algebras 
$A\otimes _RB$ and $A\; _Q\otimes B$.
\end{remark}

We recall the following concept and result from \cite{lpvo}:
\begin{proposition} \label{scheme}
Let $D$ be an algebra with multiplication denoted by
$\mu _D=\mu $ and $T:D\ot D\rightarrow D\ot D$ a linear map satisfying the
following conditions: $T(1\ot d)=1\ot d$, $T(d\ot 1)=d\ot 1$, for all
$d\in D$, and
\begin{eqnarray}
&&\mu _{23}\circ T_{13}\circ T_{12}=
T\circ \mu _{23}:D\ot D\ot D\rightarrow D\ot D, \label{dec1} \\
&&\mu _{12}\circ T_{13}\circ T_{23}=
T\circ \mu _{12}:D\ot D\ot D\rightarrow D\ot D,  \label{dec2}\\
&&T_{12}\circ T_{23}=T_{23}\circ T_{12}:
D\ot D\ot D\rightarrow D\ot D\ot D, \label{dec3}
\end{eqnarray}
with standard notation for $T_{ij}$ and $\mu _{ij}$. 
Then the bilinear map $\mu \circ T:D\ot D\rightarrow D$ is another
associative algebra structure on $D$ (with the same unit 1) denoted by 
$D^T$, and the map $T$ is 
called a {\bf twistor} for $D$.
\end{proposition}
\begin{proposition}
Let  $A\; _Q\otimes _RB$ be an L-R-twisted tensor product of algebras. 
Define the maps $T_1, T_2, T_3:
(A\otimes B)\otimes (A\otimes B)\rightarrow (A\otimes B)\otimes 
(A\otimes B)$ by
\begin{eqnarray*}
&&T_1((a\otimes b)\otimes (a'\otimes b'))=(a_Q\otimes b_R)\otimes 
(a'_R\otimes b'_Q), \\
&&T_2((a\otimes b)\otimes (a'\otimes b'))=(a_Q\otimes b)\otimes 
(a'\otimes b'_Q), \\
&&T_3((a\otimes b)\otimes (a'\otimes b'))=(a\otimes b_R)\otimes 
(a'_R\otimes b'). 
\end{eqnarray*}
Then $T_1$ is a twistor for $A\otimes B$, $T_2$ is a twistor for 
$A\otimes _RB$, $T_3$ is a twistor for 
$A\; _Q\otimes B$ and moreover we have $A\; _Q\otimes _RB=
(A\otimes B)^{T_1}=(A\otimes _RB)^{T_2}=
(A\; _Q\otimes B)^{T_3}$.
\end{proposition}
\begin{proof}
Straightforward computations, using the relations (\ref{tw0})--(\ref{tw5'}).
\end{proof}
\begin{proposition} \label{izobij}
Let $A\; _Q\otimes _RB$ be an L-R-twisted tensor product of algebras 
such that $Q$ is bijective with inverse $Q^{-1}$. 
Then the map $P:B\otimes A\rightarrow A\otimes B$, $P=Q^{-1}\circ R$, 
is a twisting map, and we have an algebra isomorphism 
$Q:A\otimes _PB\rightarrow A\; _Q\otimes _RB$. Thus, an L-R-twisted 
tensor product with bijective $Q$ is 
isomorphic to an ordinary twisted tensor product.  
\end{proposition}
\begin{proof}
First, it is obvious that $a_P\otimes 1=a\otimes 1$ and 
$1_P\otimes b_P=1\otimes b$, for all $a\in A$, $b\in B$. 
We check now that $P$ is a twisting map. Let $a, a'\in A$, $b, b'\in B$; 
we denote 
$Q^{-1}(a\otimes b)=a_{Q^{-1}}\otimes b_{Q^{-1}}=a_{q^{-1}}\otimes 
b_{q^{-1}}$ and 
$P(b\otimes a)=a_P\otimes b_P=a_p\otimes b_p$. We compute (denote $Q=q=\overline{Q}=
\overline{q}$):
\begin{eqnarray*}
Q(a_Pa'_p\otimes b_{P_p})&=&(a_{R_{Q^{-1}}}a'_{r_{q^{-1}}})_Q
\otimes (b_{R_{Q^{-1}_{r_{q^{-1}}}}})_Q\\
&\overset{(\ref{tw4'})}{=}&a_{R_{Q^{-1}_Q}}a'_{r_{q^{-1}_{\overline{Q}}}}
\otimes b_{R_{Q^{-1}_{r_{q^{-1}_{\overline{Q}_Q}}}}}\\
&=&a_{R_{Q^{-1}_Q}}a'_r\otimes b_{R_{Q^{-1}_{r_Q}}}\\
&\overset{(\ref{comb2})}{=}&a_{R_{Q^{-1}_Q}}a'_r\otimes 
b_{R_{Q^{-1}_{Q_r}}}\\ 
&=&a_Ra'_r\otimes b_{R_r}\\
&\overset{(\ref{tw4})}{=}&(aa')_R\otimes b_R,
\end{eqnarray*}
hence by applying $Q^{-1}$ we obtain $(aa')_P\otimes 
b_P=a_Pa'_p\otimes b_{P_p}$, that is (\ref{tw4}) for $P$. 
The fact that $P$ satisfies (\ref{tw5}) can be proved similarly. 
The only thing left to prove is that $Q$ is an algebra map. 
We compute: 
\begin{eqnarray*}
Q((a\otimes b)(a'\otimes b'))&=&(aa'_P)_Q\otimes (b_Pb')_Q\\
&\overset{(\ref{tw4'})}{=}&a_qa'_{P_Q}\otimes (b_Pb')_{Q_q}\\
&\overset{(\ref{tw5'})}{=}&a_qa'_{P_{Q_{\overline{Q}}}}
\otimes (b_{P_Q}b'_{\overline{Q}})_q\\
&\overset{(\ref{tw5'})}{=}&a_{q_{\overline{q}}}a'_{P_{Q_{\overline{Q}}}}
\otimes b_{P_{Q_q}}b'_{\overline{Q}_{\overline{q}}}\\
&=&a_{q_{\overline{q}}}a'_{R_{\overline{Q}}}\otimes b_{R_q}
b'_{\overline{Q}_{\overline{q}}}\\
&\overset{(\ref{comb1})}{=}&a_{q_{\overline{q}}}a'_{\overline{Q}_R}
\otimes b_{R_q}b'_{\overline{Q}_{\overline{q}}}\\
&\overset{(\ref{comb2})}{=}&a_{q_{\overline{q}}}a'_{\overline{Q}_R}
\otimes b_{q_R}b'_{\overline{Q}_{\overline{q}}}\\
&=&(a_q\otimes b_q)(a'_{\overline{Q}}\otimes b'_{\overline{Q}})\\
&=&Q(a\otimes b)Q(a'\otimes b'),
\end{eqnarray*} 
finishing the proof.
\end{proof}
\begin{remark}
Proposition \ref{izobij} generalizes (and was inspired by) the 
following result in \cite{pvo}. Let $H$ be a Hopf algebra 
with bijective antipode $S$ and ${\cal A}$ an $H$-bimodule algebra. 
Consider the so-called diagonal crossed product 
(cf. \cite{bpvo}, \cite{hn}) ${\cal A}\bowtie H$, which is an 
associative algebra built on ${\cal A}\otimes H$, with 
multiplication defined by $(\varphi \bowtie h)(\varphi '\bowtie h')=
\varphi (h_1\cdot \varphi '\cdot S^{-1}(h_3))\bowtie 
h_2h'$, for all $\varphi , \varphi '\in {\cal A}$ and $h, h'\in H$, 
that is, ${\cal A}\bowtie H$ is the twisted tensor product 
${\cal A}\otimes _PH$, where $P:H\otimes {\cal A}\rightarrow 
{\cal A}\otimes H$, $P(h\otimes \varphi )=
h_1\cdot \varphi \cdot S^{-1}(h_3)\otimes h_2$. Then the 
map $Q:{\cal A}\bowtie H\rightarrow {\cal A}\nat H$, 
$Q(\varphi \bowtie h)=\varphi \cdot h_2\nat h_1$, is an algebra 
isomorphism. 
\end{remark}

The next result generalizes the corresponding one 
for twisted tensor products (cf. \cite{bm}):
\begin{proposition}
Let $A\; _Q\ot _RB$ and $A'\; _{Q'}\ot _{R'}B'$ be two L-R-twisted 
tensor products 
of algebras and $f:A\rightarrow A'$ and $g:B\rightarrow B'$ two algebra 
maps satisfying the conditions $(f\ot g)\circ R=R'\circ (g\ot f)$ and 
$(f\ot g)\circ Q=Q'\circ (f\ot g)$. Then $f\ot g$ is an algebra map from 
$A\; _Q\ot _RB$ to $A'\; _{Q'}\ot _{R'}B'$.
\end{proposition}
\begin{proof}
It is obvious that $f\ot g$ is unital, so we only prove that it is 
multiplicative. Let $a, x\in A$ and $b, y\in B$; we compute:
\begin{eqnarray*}
(f\ot g)((a\ot b)(x\ot y))&=&(f\ot g)(a_Qx_R\ot b_Ry_Q)\\
&=&f(a_Q)f(x_R)\ot g(b_R)g(y_Q)\\
&=&f(a)_{Q'}f(x)_{R'}\ot g(b)_{R'}g(y)_{Q'}\\
&=&(f(a)\ot g(b))(f(x)\ot g(y))\\
&=&(f\ot g)(a\ot b)(f\ot g)(x\ot y), 
\end{eqnarray*}
finishing the proof. 
\end{proof}
\section{Iterated L-R-twisted tensor products of algebras}
\setcounter{equation}{0}
${\;\;\;}$
It was proved in \cite{jlpvo} that, under certain circumstances, twisted tensor 
products of algebras may be iterated. More precisely: 
\begin{theorem}(\cite{jlpvo})\label{itertw}
Let $A\ot _{R_1}B$, $B\ot _{R_2}C$ and $A\ot _{R_3}C$ be 
twisted tensor products 
of algebras. Define the maps 
\begin{eqnarray*}
&&T_1: C\ot (A\ot _{R_1}B)\rightarrow  (A\ot _{R_1}B)\ot C, \;\;\;T_1=
(id_A\ot R_2)\circ (R_3\ot id_B), \\
&&T_2: (B\ot _{R_2}C)\ot A\rightarrow A\ot  (B\ot _{R_2}C), \;\;\;T_2=
(R_1\ot id_C)\circ (id_B\ot R_3), 
\end{eqnarray*}
and assume that $R_1$, $R_2$, $R_3$ satisfy the following compatibility 
condition (the hexagon equation): 
\begin{eqnarray*}
&&(id_A\ot R_2)\circ (R_3\ot id_B)\circ (id_C\ot R_1)=
(R_1\ot id_C)\circ (id_B\ot R_3)
\circ (R_2\ot id_A).
\end{eqnarray*}
Then $T_1$ is a twisting map between $A\ot _{R_1}B$ and 
$C$, $T_2$ is a twisting map between 
$A$ and $B\ot _{R_2}C$ and moreover the algebras 
$(A\ot _{R_1}B)\ot _{T_1}C$ and 
$A\ot _{T_2}(B\ot _{R_2}C)$ coincide. 
\end{theorem} 

Our aim is now is to generalize this result for L-R-twisted tensor 
products of algebras. We begin with 
what will be our guiding example, namely a situation when L-R-smash 
products may be iterated.  

We recall the following construction introduced in \cite{panvan}.  
Let $H$ be a bialgebra and denote by ${\cal LR}(H)$ the category 
whose objects  
are vector spaces $M$ endowed with $H$-bimodule and 
$H$-bicomodule structures (denoted by $h\otimes m\mapsto h\cdot m$, 
$m\otimes h\mapsto m\cdot h$, $m\mapsto m^{(-1)}\otimes m^{(0)}$, 
$m\mapsto m^{<0>}\otimes m^{<1>}$, for all $h\in H$, $m\in M$), such that 
$M$ is a left-left Yetter-Drinfeld module, a left-right Long module, 
a right-right Yetter-Drinfeld module and a right-left Long module, i.e.  
\begin{eqnarray}
&&(h_1\cdot m)^{(-1)}h_2\otimes (h_1\cdot m)^{(0)}=
h_1m^{(-1)}\otimes h_2\cdot m^{(0)}, \label{ydl1} \\
&&(h\cdot m)^{<0>}\otimes (h\cdot m)^{<1>}=h\cdot m^{<0>}\otimes 
m^{<1>}, \label{ydl2} \\
&&(m\cdot h_2)^{<0>}\otimes h_1(m\cdot h_2)^{<1>}=
m^{<0>}\cdot h_1\otimes m^{<1>}h_2, \label{ydl3} \\
&&(m\cdot h)^{(-1)}\otimes (m\cdot h)^{(0)}=m^{(-1)}\otimes 
m^{(0)}\cdot h, \label{ydl4}
\end{eqnarray}
for all $h\in H$, $m\in M$; the morphisms in ${\cal LR}(H)$ are the 
$H$-bilinear $H$-bicolinear maps. The objects of  ${\cal LR}(H)$ are called 
Yetter-Drinfeld-Long bimodules. This category ${\cal LR}(H)$ is a 
strict monoidal category,  
with unit $k$ endowed with usual $H$-bimodule and $H$-bicomodule structures, 
and tensor product given as follows: if $M, N\in {\cal LR}(H)$ then 
$M\otimes N\in {\cal LR}(H)$ with structures (for all $m\in M$, $n\in N$, 
$h\in H$): 
\begin{eqnarray*}
&&h\cdot (m\otimes n)=h_1\cdot m\otimes h_2\cdot n, \\
&&(m\otimes n)\cdot h=m\cdot h_1\otimes n\cdot h_2, \\
&&(m\otimes n)^{(-1)}\otimes (m\otimes n)^{(0)}=m^{(-1)}n^{(-1)}\otimes 
(m^{(0)}\otimes n^{(0)}), \\
&&(m\otimes n)^{<0>}\otimes (m\otimes n)^{<1>}=(m^{<0>}\otimes n^{<0>})
\otimes m^{<1>}n^{<1>}.
\end{eqnarray*}

Let now $A$ be an algebra in the monoidal category ${\cal LR}(H)$. 
By looking at the 
definition of ${\cal LR}(H)$ as a monoidal category, it is easy to see that, 
in particular, 
$A$ is an $H$-bimodule algebra and an $H$-bicomodule algebra. 
Thus, if ${\cal A}$ is 
an $H$-bimodule algebra, we can consider the associative algebras 
${\cal A}\nat A$ and 
$A\nat H$. 
\begin{proposition}\label{iterlrsmash}
Let $H$ be a bialgebra, $A$ an algebra in ${\cal LR}(H)$ and ${\cal A}$ an 
$H$-bimodule algebra. Then:\\
(i) ${\cal A}\nat A$ is an $H$-bimodule algebra, with $H$-module structures 
given by 
$h\cdot (\varphi \nat a)=h_1\cdot \varphi \nat h_2\cdot a$ and 
$(\varphi \nat a)\cdot h=
\varphi \cdot h_2\nat a\cdot h_1$, for all $\varphi \in {\cal A}$, 
$a\in A$, $h\in H$. \\
(ii) $A\nat H$ is an $H$-bicomodule algebra, with $H$-comodule 
structures given by 
\begin{eqnarray*}
&&\lambda : A\nat H\rightarrow H\ot (A\nat H), \;\;\;\lambda (a\nat h)=
a^{(-1)}h_1\ot 
(a^{(0)}\nat h_2):=(a\nat h)_{[-1]}\ot (a\nat h)_{[0]}, \\
&&\rho :A\nat H\rightarrow (A\nat H)\ot H, \;\;\;\rho (a\nat h)=
(a^{<0>}\nat h_1)\ot h_2a^{<1>}
:=(a\nat h)_{<0>}\ot (a\nat h)_{<1>}.
\end{eqnarray*}
 (iii) The algebras $({\cal A}\nat A)\nat H$ and ${\cal A}\nat (A\nat H)$ coincide.  
\end{proposition}
\begin{proof}
(i) Obviously, ${\cal A}\nat A$ is an $H$-bimodule. 
Let $h\in H$, $\varphi , \varphi '\in {\cal A}$ and 
$a, a'\in A$. We check the left and right $H$-module algebra conditions:
\begin{eqnarray*}
(h_1\cdot (\varphi \nat a))(h_2\cdot (\varphi '\nat a'))&=&
(h_1\cdot \varphi \nat h_2\cdot a)(h_3\cdot \varphi '\nat h_4\cdot a')\\
&=&((h_1\cdot \varphi )\cdot (h_4\cdot a')^{<1>})
((h_2\cdot a)^{(-1)}\cdot (h_3\cdot \varphi '))\\
&&\nat (h_2\cdot a)^{(0)}(h_4\cdot a')^{<0>}\\
&\overset{(\ref{ydl2})}{=}&(h_1\cdot \varphi \cdot a'^{<1>})
((h_2\cdot a)^{(-1)}h_3\cdot 
\varphi ')\nat (h_2\cdot a)^{(0)}(h_4\cdot a'^{<0>})\\
&\overset{(\ref{ydl1})}{=}&(h_1\cdot \varphi \cdot a'^{<1>})
(h_2a^{(-1)}\cdot 
\varphi ')\nat (h_3\cdot a^{(0)})(h_4\cdot a'^{<0>})\\
&=&h\cdot ((\varphi \cdot a'^{<1>})(a^{(-1)}\cdot \varphi ')
\nat a^{(0)}a'^{<0>})\\
&=&h\cdot ((\varphi \nat a)(\varphi '\nat a')), \;\;\;q.e.d.
\end{eqnarray*} 
\begin{eqnarray*}
((\varphi \nat a)\cdot h_1)((\varphi '\nat a')\cdot h_2)&=&
(\varphi \cdot h_2\nat a\cdot h_1)
(\varphi '\cdot h_4\nat a'\cdot h_3)\\
&=&((\varphi \cdot h_2)\cdot (a'\cdot h_3)^{<1>})
((a\cdot h_1)^{(-1)}\cdot 
(\varphi '\cdot h_4))\\
&&\nat (a\cdot h_1)^{(0)}(a'\cdot h_3)^{<0>}\\
&\overset{(\ref{ydl4})}{=}&(\varphi \cdot h_2(a'\cdot h_3)^{<1>})
(a^{(-1)}\cdot 
\varphi '\cdot h_4)\nat (a^{(0)}\cdot h_1)(a'\cdot h_3)^{<0>}\\
&\overset{(\ref{ydl3})}{=}&(\varphi \cdot a'^{<1>}h_3)(a^{(-1)}\cdot 
\varphi '\cdot h_4)\nat (a^{(0)}\cdot h_1)(a'^{<0>}\cdot h_2)\\
&=&((\varphi \cdot a'^{<1>})(a^{(-1)}\cdot \varphi '))\cdot h_2
\nat (a^{(0)}a'^{<0>})\cdot h_1\\
&=&((\varphi \cdot a'^{<1>})(a^{(-1)}\cdot \varphi ')\nat a^{(0)}
a'^{<0>})\cdot h\\
&=&((\varphi \nat a)(\varphi '\nat a'))\cdot h, \;\;\;q.e.d.
\end{eqnarray*}
(ii) It is very easy to see that $A\nat H$ is an $H$-bicomodule, 
so we check the left and 
right $H$-comodule algebra conditions:
\begin{eqnarray*}
\lambda ((a\nat h)(a'\nat h'))&=&\lambda ((a\cdot h'_2)(h_1\cdot a')
\nat h_2h'_1)\\
&=&(a\cdot h'_3)^{(-1)}(h_1\cdot a')^{(-1)}h_2h'_1\ot 
((a\cdot h'_3)^{(0)}(h_1\cdot a')^{(0)}\nat h_3h'_2)\\
&\overset{(\ref{ydl4})}{=}&a^{(-1)}(h_1\cdot a')^{(-1)}h_2h'_1\ot 
((a^{(0)}\cdot h'_3)(h_1\cdot a')^{(0)}\nat h_3h'_2)\\
&\overset{(\ref{ydl1})}{=}&a^{(-1)}h_1a'^{(-1)}h'_1\ot 
((a^{(0)}\cdot h'_3)(h_2\cdot a'^{(0)})\nat h_3h'_2)\\
&=&(a^{(-1)}h_1\ot (a^{(0)}\nat h_2))(a'^{(-1)}h'_1\ot (a'^{(0)}\nat h'_2))\\
&=&\lambda (a\nat h)\lambda (a'\nat h'), \;\;\;q.e.d.
\end{eqnarray*}
\begin{eqnarray*}
\rho ((a\nat h)(a'\nat h'))&=&\rho  ((a\cdot h'_2)(h_1\cdot a')\nat h_2h'_1)\\
&=&((a\cdot h'_3)^{<0>}(h_1\cdot a')^{<0>}\nat h_2h'_1)\ot 
h_3h'_2(a\cdot h'_3)^{<1>}(h_1\cdot a')^{<1>}\\
&\overset{(\ref{ydl2})}{=}&((a\cdot h'_3)^{<0>}(h_1\cdot a'^{<0>})
\nat h_2h'_1)\ot 
h_3h'_2(a\cdot h'_3)^{<1>}a'^{<1>}\\
&\overset{(\ref{ydl3})}{=}&((a^{<0>}\cdot h'_2)(h_1\cdot a'^{<0>})
\nat h_2h'_1)\ot 
h_3a^{<1>}h'_3a'^{<1>}\\
&=&(a^{<0>}\nat h_1)(a'^{<0>}\nat h'_1)\ot h_2a^{<1>}h'_2a'^{<1>}\\
&=&\rho (a\nat h)\rho (a'\nat h'), \;\;\;q.e.d.
\end{eqnarray*}
(iii) We write down the multiplication of $({\cal A}\nat A)\nat H$:
\begin{eqnarray*}
((\varphi \nat a)\nat h)((\varphi ' \nat a')\nat h')&=&((\varphi \nat a)\cdot h'_2)
(h_1\cdot (\varphi '\nat a'))\nat h_2h'_1\\
&=&(\varphi \cdot h'_3\nat a\cdot h'_2)(h_1\cdot \varphi '\nat h_2\cdot a')
\nat h_3h'_1\\
&=&(\varphi \cdot h'_3(h_2\cdot a')^{<1>})((a\cdot h'_2)^{(-1)}h_1\cdot 
\varphi ')\\
&&\nat 
(a\cdot h'_2)^{(0)}(h_2\cdot a')^{<0>}\nat h_3h'_1\\
&\overset{(\ref{ydl2}), (\ref{ydl4})}{=}&(\varphi \cdot h'_3a'^{<1>})
(a^{(-1)}h_1\cdot \varphi ')
\nat (a^{(0)}\cdot h'_2)(h_2\cdot a'^{<0>})\nat h_3h'_1. 
\end{eqnarray*}
We write down the multiplication of ${\cal A}\nat (A\nat H)$:
\begin{eqnarray*}
(\varphi \nat (a\nat h))(\varphi ' \nat (a'\nat h'))&=&
(\varphi \cdot (a'\nat h')_{<1>})((a\nat h)_{[-1]}\cdot \varphi ')
\nat (a\nat h)_{[0]}
(a'\nat h')_{<0>}\\
&=&(\varphi \cdot h'_2a'^{<1>})(a^{(-1)}h_1\cdot \varphi ')
\nat (a^{(0)}\nat h_2)(a'^{<0>}\nat h'_1)\\
&=&(\varphi \cdot h'_3a'^{<1>})(a^{(-1)}h_1\cdot \varphi ')
\nat (a^{(0)}\cdot h'_2)
(h_2\cdot a'^{<0>})\nat h_3h'_1,
\end{eqnarray*}
and we see that the two multiplications coincide.
\end{proof}

We are able now to find a common generalization of Theorem \ref{itertw} and 
Proposition \ref{iterlrsmash}:
\begin{theorem}
Let $A\;_{Q_1}\ot _{R_1}B$,  $B\;_{Q_2}\ot _{R_2}C$,  
$A\;_{Q_3}\ot _{R_3}C$ be three 
L-R-twisted tensor products of algebras, such that the following 
conditions are satisfied, 
for all $a\in A$, $b\in B$, $c\in C$:
\begin{eqnarray}
&&(a_{R_1})_{R_3}\ot (b_{R_1})_{R_2}\ot (c_{R_3})_{R_2}=
(a_{R_3})_{R_1}\ot (b_{R_2})_{R_1}\ot (c_{R_2})_{R_3}, \label{YB} \\ 
&&(a_{Q_1})_{Q_3}\ot (b_{Q_1})_{Q_2}\ot (c_{Q_3})_{Q_2}=
(a_{Q_3})_{Q_1}\ot (b_{Q_2})_{Q_1}\ot (c_{Q_2})_{Q_3}, \label{YBQuri} \\
&&a_{R_1}\ot (b_{R_1})_{Q_2}\ot c_{Q_2}=
a_{R_1}\ot (b_{Q_2})_{R_1}\ot c_{Q_2}, \label{comb3} \\
&&a_{Q_1}\ot (b_{R_2})_{Q_1}\ot c_{R_2}=
a_{Q_1}\ot (b_{Q_1})_{R_2}\ot c_{R_2}, \label{comb4} \\
&&(a_{Q_1})_{R_3}\ot b_{Q_1}\ot c_{R_3}=
(a_{R_3})_{Q_1}\ot b_{Q_1}\ot c_{R_3}, \label{comb5} \\
&&(a_{R_1})_{Q_3}\ot b_{R_1}\ot c_{Q_3}=
(a_{Q_3})_{R_1}\ot b_{R_1}\ot c_{Q_3}, \label{comb6} \\
&&a_{R_3}\ot b_{Q_2}\ot (c_{Q_2})_{R_3}=
a_{R_3}\ot b_{Q_2}\ot (c_{R_3})_{Q_2}, \label{comb7} \\
&&a_{Q_3}\ot b_{R_2}\ot (c_{Q_3})_{R_2}=
a_{Q_3}\ot b_{R_2}\ot (c_{R_2})_{Q_3}. \label{comb8} 
\end{eqnarray}
Define the maps 
\begin{eqnarray*}
&&T_1:C\ot (A\ot B)\rightarrow (A\ot B)\ot C, \;\;\;T_1(c\ot (a\ot b))=
(a_{R_3}\ot b_{R_2})\ot (c_{R_3})_{R_2}, \\
&&V_1:(A\ot B)\ot C\rightarrow (A\ot B)\ot C, \;\;\;V_1((a\ot b)\ot c)=
(a_{Q_3}\ot b_{Q_2})\ot (c_{Q_3})_{Q_2}, \\
&&T_2:(B\ot C)\ot A\rightarrow A\ot (B\ot C), \;\;\;T_2((b\ot c)\ot a)=
(a_{R_3})_{R_1}\ot (b_{R_1}\ot c_{R_3}), \\
&&V_2:A\ot (B\ot C)\rightarrow A\ot (B\ot C), \;\;\;V_2(a\ot (b\ot c))=
(a_{Q_3})_{Q_1}\ot (b_{Q_1}\ot c_{Q_3}). 
\end{eqnarray*}
Then $(A\;_{Q_1}\ot _{R_1}B)\;_{V_1}\ot _{T_1}C$ and 
$A\;_{V_2}\ot _{T_2}(B\;_{Q_2}\ot _{R_2}C)$ are L-R-twisted 
tensor products of algebras
 and moreover they coincide as algebras. 
\end{theorem}
\begin{proof}
We only give the proof for $(A\;_{Q_1}\ot _{R_1}B)\;_{V_1}\ot _{T_1}C$, 
the one for $A\;_{V_2}\ot _{T_2}(B\;_{Q_2}\ot _{R_2}C)$ is similar 
and left to the reader. We need to prove the relations 
(\ref{tw0})--(\ref{comb2}) for the maps $T_1$ and $V_1$. We only 
prove (\ref{tw4}), (\ref{tw4'}) and (\ref{comb2}), the other relations 
are very easy to prove and are left to the reader. \\
\underline{Proof of (\ref{tw4}):}
\begin{eqnarray*}
((a\ot b)(a'\ot b'))_{T_1}\ot c_{T_1}&=&(a_{Q_1}a'_{R_1}\ot 
b_{R_1}b'_{Q_1})_{T_1}\ot c_{T_1}\\
&=&(a_{Q_1}a'_{R_1})_{R_3}\ot 
(b_{R_1}b'_{Q_1})_{R_2}\ot (c_{R_3})_{R_2}\\
&\overset{(\ref{tw4})}{=}&(a_{Q_1})_{R_3}(a'_{R_1})_{r_3}\ot 
(b_{R_1}b'_{Q_1})_{R_2}\ot ((c_{R_3})_{r_3})_{R_2}\\
&\overset{(\ref{tw4})}{=}&(a_{Q_1})_{R_3}(a'_{R_1})_{r_3}\ot 
(b_{R_1})_{R_2}(b'_{Q_1})_{r_2}\ot (((c_{R_3})_{r_3})_{R_2})_{r_2}\\
&\overset{(\ref{comb4})}{=}&(a_{Q_1})_{R_3}(a'_{R_1})_{r_3}\ot 
(b_{R_1})_{R_2}(b'_{r_2})_{Q_1}\ot (((c_{R_3})_{r_3})_{R_2})_{r_2}\\
&\overset{(\ref{comb5})}{=}&(a_{R_3})_{Q_1}(a'_{R_1})_{r_3}\ot 
(b_{R_1})_{R_2}(b'_{r_2})_{Q_1}\ot (((c_{R_3})_{r_3})_{R_2})_{r_2}\\
&\overset{(\ref{YB})}{=}&(a_{R_3})_{Q_1}(a'_{r_3})_{R_1}\ot 
(b_{R_2})_{R_1}(b'_{r_2})_{Q_1}\ot (((c_{R_3})_{R_2})_{r_3})_{r_2}\\
&=&(a_{R_3}\ot b_{R_2})(a'_{r_3}\ot b'_{r_2})
\ot (((c_{R_3})_{R_2})_{r_3})_{r_2}\\
&=&(a\ot b)_{T_1}(a'\ot b')_{t_1}\ot (c_{T_1})_{t_1}, \;\;\;q.e.d.
\end{eqnarray*} 
\underline{Proof of (\ref{tw4'}):}
\begin{eqnarray*}
((a\ot b)(a'\ot b'))_{V_1}\ot c_{V_1}&=&(a_{Q_1}a'_{R_1}\ot 
b_{R_1}b'_{Q_1})_{V_1}\ot c_{V_1}\\
&=&(a_{Q_1}a'_{R_1})_{Q_3}\ot 
(b_{R_1}b'_{Q_1})_{Q_2}\ot (c_{Q_3})_{Q_2}\\
&\overset{(\ref{tw4'})}{=}&(a_{Q_1})_{q_3}(a'_{R_1})_{Q_3}\ot 
(b_{R_1}b'_{Q_1})_{Q_2}\ot ((c_{Q_3})_{q_3})_{Q_2}\\
&\overset{(\ref{tw4'})}{=}&(a_{Q_1})_{q_3}(a'_{R_1})_{Q_3}\ot 
(b_{R_1})_{q_2}(b'_{Q_1})_{Q_2}\ot (((c_{Q_3})_{q_3})_{Q_2})_{q_2}\\
&\overset{(\ref{YBQuri})}{=}&(a_{q_3})_{Q_1}(a'_{R_1})_{Q_3}\ot 
(b_{R_1})_{q_2}(b'_{Q_2})_{Q_1}\ot (((c_{Q_3})_{Q_2})_{q_3})_{q_2}\\
&\overset{(\ref{comb3})}{=}&(a_{q_3})_{Q_1}(a'_{R_1})_{Q_3}\ot 
(b_{q_2})_{R_1}(b'_{Q_2})_{Q_1}\ot (((c_{Q_3})_{Q_2})_{q_3})_{q_2}\\
&\overset{(\ref{comb6})}{=}&(a_{q_3})_{Q_1}(a'_{Q_3})_{R_1}\ot 
(b_{q_2})_{R_1}(b'_{Q_2})_{Q_1}\ot (((c_{Q_3})_{Q_2})_{q_3})_{q_2}\\
&=&(a_{q_3}\ot b_{q_2})(a'_{Q_3}\ot b'_{Q_2})\ot  
(((c_{Q_3})_{Q_2})_{q_3})_{q_2}\\
&=&(a\ot b)_{v_1}(a'\ot b')_{V_1}\ot (c_{V_1})_{v_1}, \;\;\;q.e.d.
\end{eqnarray*}
\underline{Proof of (\ref{comb2}):}
\begin{eqnarray*}
(a\ot b)_{T_1}\ot (c_{V_1})_{T_1}\ot (a'\ot b')_{V_1}&=&
(a_{R_3}\ot b_{R_2})\ot (((c_{Q_3})_{Q_2})_{R_3})_{R_2}\ot 
(a'_{Q_3}\ot b'_{Q_2})\\
&\overset{(\ref{comb7})}{=}&(a_{R_3}\ot b_{R_2})\ot 
(((c_{Q_3})_{R_3})_{Q_2})_{R_2}\ot 
(a'_{Q_3}\ot b'_{Q_2})\\
&\overset{(\ref{comb2})}{=}&(a_{R_3}\ot b_{R_2})\ot 
(((c_{R_3})_{Q_3})_{R_2})_{Q_2}\ot 
(a'_{Q_3}\ot b'_{Q_2})\\
&\overset{(\ref{comb8})}{=}&(a_{R_3}\ot b_{R_2})\ot 
(((c_{R_3})_{R_2})_{Q_3})_{Q_2}\ot 
(a'_{Q_3}\ot b'_{Q_2})\\
&=&(a\ot b)_{T_1}\ot (c_{T_1})_{V_1}\ot (a'\ot b')_{V_1}, \;\;\;q.e.d.
\end{eqnarray*}
We prove now that  $(A\;_{Q_1}\ot _{R_1}B)\;_{V_1}\ot _{T_1}C\equiv  
A\;_{V_2}\ot _{T_2}(B\;_{Q_2}\ot _{R_2}C)$. We write down 
the multiplication of  
 $(A\;_{Q_1}\ot _{R_1}B)\;_{V_1}\ot _{T_1}C$:
\begin{eqnarray*}
((a\ot b)\ot c)((a'\ot b')\ot c')&=&(a\ot b)_{V_1}(a'\ot b')_{T_1}\ot 
c_{T_1}c'_{V_1}\\
&=&(a_{Q_3}\ot b_{Q_2})(a'_{R_3}\ot b'_{R_2})\ot 
(c_{R_3})_{R_2}(c'_{Q_3})_{Q_2}\\
&=&(a_{Q_3})_{Q_1}(a'_{R_3})_{R_1}\ot 
(b_{Q_2})_{R_1}(b'_{R_2})_{Q_1}\ot 
(c_{R_3})_{R_2}(c'_{Q_3})_{Q_2}\\
&\overset{(\ref{comb4})}{=}&(a_{Q_3})_{Q_1}(a'_{R_3})_{R_1}\ot 
(b_{Q_2})_{R_1}(b'_{Q_1})_{R_2}\ot 
(c_{R_3})_{R_2}(c'_{Q_3})_{Q_2}.
\end{eqnarray*}
 We write down the multiplication of  $A\;_{V_2}\ot 
_{T_2}(B\;_{Q_2}\ot _{R_2}C)$:
\begin{eqnarray*}
(a\ot (b\ot c))(a'\ot (b'\ot c'))&=&a_{V_2}a'_{T_2}\ot (b\ot c)_{T_2}
(b'\ot c')_{V_2}\\
&=&(a_{Q_3})_{Q_1}(a'_{R_3})_{R_1}\ot (b_{R_1}\ot c_{R_3})
(b'_{Q_1}\ot c'_{Q_3})\\
&=&(a_{Q_3})_{Q_1}(a'_{R_3})_{R_1}\ot 
(b_{R_1})_{Q_2}(b'_{Q_1})_{R_2}
\ot (c_{R_3})_{R_2}(c'_{Q_3})_{Q_2}\\
&\overset{(\ref{comb3})}{=}&(a_{Q_3})_{Q_1}(a'_{R_3})_{R_1}\ot 
(b_{Q_2})_{R_1}(b'_{Q_1})_{R_2}
\ot (c_{R_3})_{R_2}(c'_{Q_3})_{Q_2}, 
\end{eqnarray*}
and we can see that the two formulae are identical.
\end{proof}
\section{Invariance under twisting}
\setcounter{equation}{0}
${\;\;\;}$
Let $H$ be a bialgebra and $F\in H\ot H$ a 2-cocycle, that is $F$ is
invertible and satisfies
\begin{eqnarray*}
&&(\varepsilon \ot id)(F)=(id \ot \varepsilon )(F)=1, \\
&&(1\ot F)(id\ot \Delta )(F)=(F\ot 1)(\Delta \ot id)(F).
\end{eqnarray*}
We denote $F=F^1\ot F^2$ and $F^{-1}=G^1\ot G^2$. We denote by $H_F$
the Drinfeld twist of $H$, which is a bialgebra having the same
algebra structure as $H$ and comultiplication given by $\Delta
_F(h)= F\Delta (h)F^{-1}$, for all $h\in H$.

If $A$ is a left $H$-module algebra (with $H$-action denoted by
$h\ot a\mapsto h\cdot a$), the invariance under twisting of the
smash product $A\# H$ is the following result (see \cite{Majid97a},
\cite{Bulacu00a}). Define a new multiplication on $A$, by
$a*a'=(G^1\cdot a)(G^2\cdot a')$, for all $a, a'\in A$, and denote
by $A_{F^{-1}}$ the new structure; then $A_{F^{-1}}$ is a left
$H_F$-module algebra (with the same action as for $A$) and we have
an algebra isomorphism 
$A_{F^{-1}}\# H_F\simeq A\# H, \;\;a\# h\mapsto G^1\cdot a\# G^2h$.

This result was regarded in \cite{jlpvo} as a particular case of a very 
general result (Theorem 4.4) for twisted tensor products of algebras, 
that was called ''invariance under twisting'' for twisted tensor products 
of algebras.

Let again $H$ be a bialgebra, $F\in H\ot H$ a 2-cocycle and ${\cal A}$ 
an $H$-bimodule algebra. Define a new multiplication on ${\cal A}$, by 
$\varphi \bullet \varphi '=(G^1\cdot \varphi \cdot F^1)(G^2\cdot 
\varphi '\cdot F^2)$, where $F=F^1\otimes F^2$ and $F^{-1}=G^1\otimes G^2$, 
and denote by $_F{\cal A}_{F^{-1}}$ the new structure. Then one can 
easily see that $_F{\cal A}_{F^{-1}}$ is an $H_F$-bimodule algebra, 
and moreover we have the following invariance under twisting for 
L-R-smash products: 
\begin{proposition}\label{invutwlrsmash}
We have an algebra isomorphism 
\begin{eqnarray*}
&&_F{\cal A}_{F^{-1}}\nat H_F\simeq 
{\cal A}\nat H, \;\;\; \varphi \nat h\mapsto G^1\cdot \varphi \cdot F^2\nat 
G^2hF^1.
\end{eqnarray*} 
\end{proposition}

Our aim  is to prove an ''invariance under twisting'' for 
L-R-twisted tensor products of algebras, that is, to find a common 
generalization of Proposition \ref{invutwlrsmash} and of Theorem 4.4 in 
\cite{jlpvo}. 
\begin{proposition} \label{pregat}
Let $A, B$ be two algebras and $R:B\ot A\rightarrow A\ot B$, 
$Q:A\otimes B\rightarrow A\otimes B$ two linear
maps, with notation $R(b\ot a)=a_R\ot b_R$ and 
$Q(a\otimes b)=a_Q\otimes b_Q$, for all $a\in A$ and
$b\in B$, such that (\ref{comb1}) holds.  
Assume that we are given linear maps 
$\mu _l:B\ot A\rightarrow A$, $b\ot a\mapsto b\cdot a$,  
$\mu _r:A\otimes B\rightarrow A$, $a\otimes b\mapsto a\cdot b$, 
$\rho _r:A\rightarrow A\ot B$, $\rho _r(a)= a_{(0)}\ot a_{(1)}$, 
$\rho _l:A\rightarrow B\otimes A$, $\rho _l(a)=a_{<-1>}\ot a_{<0>}$, 
and denote $a\bullet a':=(a_{(0)}\cdot a'_{<-1>})(a_{(1)}\cdot a'_{<0>})$, 
for all $a, a'\in A$. Assume that
the following conditions are satisfied: $\rho _r(1)=1\ot 1$, 
$\rho _l(1)=1\ot 1$, $1\cdot a=a=a\cdot 1$, 
$a_{(0)}(a_{(1)}\cdot 1)=a$, $(1\cdot a_{<-1>})a_{<0>}=a$, and
\begin{eqnarray}
&&b\cdot (a_{(0)}(a_{(1)}\cdot a'))=a_{(0)_R}(b_Ra_{(1)}\cdot a'), 
\label{4.2}\\
&&((a\cdot a'_{<-1>})a'_{<0>})\cdot b=(a\cdot a'_{<-1>}b_Q)a'_{<0>_Q}, 
\label{sup1} \\
&&\rho _r(a\bullet a')=(a_{(0)}\cdot a'_{(0)_{R_{<-1>}}})
a'_{(0)_{R_{<0>}}}\ot a_{(1)_R}a'_{(1)}, \label{sup2} \\
&&\rho _l(a\bullet a')=a_{<-1>}a'_{<-1>_Q}\otimes a_{<0>_{Q_{(0)}}}
(a_{<0>_{Q_{(1)}}}\cdot a'_{<0>}), \label{sup3}\\
&&a_{(0)_{<-1>}}\otimes a_{(0)_{<0>}}\otimes a_{(1)}=a_{<-1>}\otimes 
a_{<0>_{(0)}}\otimes a_{<0>_{(1)}}, \label{sup4} \\
&&a_{Q_{(0)}}\otimes a_{Q_{(1)}}\otimes b_Q=a_{(0)_Q}\otimes a_{(1)}
\otimes b_Q, \label{sup5} \\
&&a_{R_{<-1>}}\otimes a_{R_{<0>}}\otimes b_R=a_{<-1>}\otimes 
a_{<0>_R}\otimes b_R, \label{sup6}
\end{eqnarray}
for all $a, a'\in A$ and $b\in B$. Then $(A, \bullet , 1)$ is an
associative unital algebra, denoted in what follows by $\tilde{A}$.
\end{proposition}
\begin{proof}
Obviously, $1$ is the unit, so we only prove the associativity of 
$\bullet $; we compute: 
\begin{eqnarray*}
(a\bullet a')\bullet a''&=&((a\bullet a')_{(0)}\cdot a''_{<-1>})
((a\bullet a')_{(1)}\cdot a''_{<0>})\\
&\overset{(\ref{sup2})}{=}&[((a_{(0)}\cdot a'_{(0)_{R_{<-1>}}})
a'_{(0)_{R_{<0>}}})\cdot a''_{<-1>}][a_{(1)_R}a'_{(1)}\cdot a''_{<0>}]\\
&\overset{(\ref{sup1})}{=}&(a_{(0)}\cdot a'_{(0)_{R_{<-1>}}}a''_{<-1>_Q})
a'_{(0)_{R_{<0>_Q}}}(a_{(1)_R}a'_{(1)}\cdot a''_{<0>})\\
&\overset{(\ref{sup6})}{=}&(a_{(0)}\cdot a'_{(0)_{<-1>}}a''_{<-1>_Q})
a'_{(0)_{<0>_{R_Q}}}(a_{(1)_R}a'_{(1)}\cdot a''_{<0>})\\
&\overset{(\ref{sup4})}{=}&(a_{(0)}\cdot a'_{<-1>}a''_{<-1>_Q})
a'_{<0>_{(0)_{R_Q}}}(a_{(1)_R}a'_{<0>_{(1)}}\cdot a''_{<0>}), 
\end{eqnarray*}
\begin{eqnarray*}
a\bullet (a'\bullet a'')&=&(a_{(0)}\cdot (a'\bullet a'')_{<-1>})
(a_{(1)}\cdot (a'\bullet a'')_{<0>})\\
&\overset{(\ref{sup3})}{=}&(a_{(0)}\cdot a'_{<-1>}a''_{<-1>_Q})
(a_{(1)}\cdot (a'_{<0>_{Q_{(0)}}}(a'_{<0>_{Q_{(1)}}}\cdot a''_{<0>})))\\
&\overset{(\ref{4.2})}{=}&(a_{(0)}\cdot a'_{<-1>}a''_{<-1>_Q})
a'_{<0>_{Q_{(0)_R}}}(a_{(1)_R}a'_{<0>_{Q_{(1)}}}\cdot a''_{<0>})\\
&\overset{(\ref{sup5})}{=}&(a_{(0)}\cdot a'_{<-1>}a''_{<-1>_Q})
a'_{<0>_{(0)_{Q_R}}}(a_{(1)_R}a'_{<0>_{(1)}}\cdot a''_{<0>})\\
&\overset{(\ref{comb1})}{=}&(a_{(0)}\cdot a'_{<-1>}a''_{<-1>_Q})
a'_{<0>_{(0)_{R_Q}}}(a_{(1)_R}a'_{<0>_{(1)}}\cdot a''_{<0>}),
\end{eqnarray*}
finishing the proof.
\end{proof}
\begin{theorem} \label{invundtw}
Assume that the hypotheses of Proposition \ref{pregat} are satisfied,
such that moreover $A\; _Q\otimes _RB$ is an L-R-twisted tensor product of 
algebras. Assume also that we are
given linear maps
$\lambda _r:A\rightarrow A\ot B$, $\lambda _r(a)=a_{[0]}\ot a_{[1]}$, and 
$\lambda _l:A\rightarrow B\otimes A$, $\lambda _l(a)=
a_{\{-1\}}\otimes a_{\{0\}}$,  
such that $\lambda _r(1)=1\ot 1$, $\lambda _l(1)=1\ot 1$ and
the following relations hold:
\begin{eqnarray}
&&a_{(0)_{[0]}}\ot a_{(0)_{[1]}}a_{(1)}=a\ot 1, \label{4.8}\\
&&a_{[0]_{(0)}}\ot a_{[0]_{(1)}}a_{[1]}=a\ot 1, \label{4.9}\\
&&a_{<-1>}a_{<0>_{\{-1\}}}\ot a_{<0>_{\{0\}}}=1\ot a, \label{extra1}\\
&&a_{\{-1\}}a_{\{0\}_{<-1>}}\ot a_{\{0\}_{<0>}}=1\ot a, \label{extra2}\\
&&a_{[0]_{\{-1\}}}\ot a_{[0]_{\{0\}}}\ot a_{[1]}=
a_{\{-1\}}\ot a_{\{0\}_{[0]}}\ot a_{\{0\}_{[1]}}, \label{extra6}\\
&&a_{[0]_{<-1>}}\ot a_{[0]_{<0>}}\ot a_{[1]}=
a_{<-1>}\ot a_{<0>_{[0]}}\ot a_{<0>_{[1]}}, \label{extra11}\\
&&a_{(0)_{\{-1\}}}\ot a_{(0)_{\{0\}}}\ot a_{(1)}=
a_{\{-1\}}\ot a_{\{0\}_{(0)}}\ot a_{\{0\}_{(1)}}, \label{extra13}\\
&&\lambda _r(aa')=a_{[0]_{(0)}}(a_{[0]_{(1)}}\cdot a'_{R_{[0]}})\ot 
a'_{R_{[1]}}a_{[1]_R}, \label{4.7}\\
&&\lambda _l(aa')=a'_{\{-1\}_Q}a_{Q_{\{-1\}}}\ot (a_{Q_{\{0\}}}\cdot 
a'_{\{0\}_{<-1>}})a'_{\{0\}_{<0>}}, \label{extra3} \\
&&\rho _l((a\cdot a'_{<-1>})a'_{<0>})=a_{<-1>}a'_{<-1>_Q}\ot 
a_{<0>_Q}a'_{<0>}, \label{extra4}\\
&&a_{(0)_{<-1>_R}}a'_{(1)_Q}\ot a_{(0)_{<0>_Q}}\ot a'_{(0)_{<0>_R}}\ot 
a'_{(0)_{<-1>}}\ot a_{(1)}\nonumber \\
&&\;\;\;\;\;\;\;\;\;\;=a'_{(1)}a_{(0)_{<-1>}}\ot a_{(0)_{<0>}}\ot 
a'_{(0)_{<0>}}\ot a'_{(0)_{<-1>}}\ot a_{(1)}, \label{extra5}\\
&&a'_{<0>_{Q_{\{-1\}_R}}}a_{(1)}\ot a'_{<0>_{Q_{\{0\}}}}\ot a'_{<-1>}\ot 
a_{(0)_R}\ot b_Q\nonumber \\
&&\;\;\;\;\;\;\;\;\;\;=a_{(1)_q}a'_{<0>_{Q_{q_{\{-1\}}}}}\ot 
a'_{<0>_{Q_{q_{\{0\}}}}}\ot a'_{<-1>}\ot a_{(0)}\ot b_Q, \label{extra7}\\
&&a_{(0)_{R_{[0]}}}\ot a'_{<-1>}a_{(0)_{R_{[1]_Q}}}\ot a'_{<0>_Q}\ot 
a_{(1)}\ot b_R\nonumber \\
&&\;\;\;\;\;\;\;\;\;\;
=a_{(0)_{r_{R_{[0]}}}}\ot a_{(0)_{r_{R_{[1]}}}}a'_{<-1>_R}
\ot a'_{<0>}\ot a_{(1)}\ot b_r, \label{extra8}\\
&&a_{<0>_{Q_{\{0\}_{(0)_R}}}}\ot a_{<0>_{Q_{\{0\}_{(1)}}}}\ot a_{<-1>}\ot 
a_{<0>_{Q_{\{-1\}}}}\ot b_R\ot b'_Q\nonumber \\
&&\;\;\;\;\;\;\;\;\;\;=a_{(0)_{R_{<0>_{Q_{\{0\}}}}}}\ot a_{(1)}\ot 
a_{(0)_{R_{<-1>}}}\ot a_{(0)_{R_{<0>_{Q_{\{-1\}}}}}}\ot b_R\ot b'_Q, 
\label{extra9} \\
&&a_{(0)_{R_{[1]}}}\ot a_{(0)_{R_{[0]_{<0>_Q}}}}\ot a_{(0)_{R_{[0]_{<-1>}}}}
\ot a_{(1)}\ot b_R\ot b'_Q\nonumber \\
&&\;\;\;\;\;\;\;\;\;\;=
a_{(0)_{<0>_{R_{Q_{[1]}}}}}\ot a_{(0)_{<0>_{R_{Q_{[0]}}}}}\ot 
a_{(0)_{<-1>}}\ot a_{(1)}\ot b_R\ot b'_Q, \label{extra10}\\
&&(a_{(0)}\cdot b')_{R_{[0]_{(0)}}}\ot (a_{(0)}\cdot b')_{R_{[0]_{(1)}}}
\ot (a_{(0)}\cdot b')_{R_{[1]}}\ot b_R\ot a_{(1)}\nonumber \\
&&\;\;\;\;\;\;\;\;\;\;=a_{(0)_{R_{[0]_{(0)}}}}\cdot b'\ot 
a_{(0)_{R_{[0]_{(1)}}}}\ot a_{(0)_{R_{[1]}}}\ot b_R\ot a_{(1)}, 
\label{extra12}\\
&&a'_{<-1>}\ot (b\cdot a'_{<0>})_{Q_{\{-1\}}}\ot 
(b\cdot a'_{<0>})_{Q_{\{0\}_{<-1>}}}\ot 
(b\cdot a'_{<0>})_{Q_{\{0\}_{<0>}}}\ot b'_Q\nonumber \\
&&\;\;\;\;\;\;\;\;\;\;=a'_{<-1>}\ot a'_{<0>_{Q_{\{-1\}}}}\ot 
a'_{<0>_{Q_{\{0\}_{<-1>}}}}\ot b\cdot a'_{<0>_{Q_{\{0\}_{<0>}}}}
\ot b'_Q, \label{extra14}
\end{eqnarray}
for all $a, a'\in A$ and $b, b'\in B$. Define the maps
\begin{eqnarray}
&&\tilde{R}:B\ot \tilde{A}\rightarrow \tilde{A}\ot B, \;\;\; \tilde{R}(b\ot  
a)=a_{(0)_{R_{[0]}}}\ot a_{(0)_{R_{[1]}}}b_Ra_{(1)}, \label{4.10} \\
&&\tilde{Q}:\tilde{A}\ot B\rightarrow \tilde{A}\ot B, \;\;\;
\tilde{Q}(a\ot b)=a_{<0>_{Q_{\{0\}}}}\ot a_{<-1>}b_Qa_{<0>_{Q_{\{-1\}}}}.
\label{Qtilde}
\end{eqnarray}
Then $\tilde{A}\; _{\tilde{Q}}\ot _{\tilde{R}}B$ is an L-R-twisted tensor 
product of algebras and we have an algebra isomorphism
\begin{eqnarray*}
&&\tilde{A}\; _{\tilde{Q}}\ot _{\tilde{R}}B\simeq A\; _Q\ot _RB, 
\;\;a\ot b\mapsto a_{(0)_{<0>}}\ot
a_{(1)}ba_{(0)_{<-1>}}.
\end{eqnarray*} 
\end{theorem}
\begin{proof}
We have to prove that $\tilde{R}$ and $\tilde{Q}$ satisfy (\ref{tw0})--
(\ref{comb2}) for the algebras $\tilde{A}$ and $B$. We will only prove 
(\ref{tw4}), (\ref{tw4'}), (\ref{comb1}), (\ref{comb2}), while 
(\ref{tw0}), (\ref{tw0'}), (\ref{tw5}), (\ref{tw5'}) are much easier and are  
left to the reader.\\ 
\underline{Proof of (\ref{tw4})}
\begin{eqnarray*}
(a\bullet a')_{\tilde{R}}\ot b_{\tilde{R}}&=&(a\bullet a')_{(0)_{R_{[0]}}}
\ot (a\bullet a')_{(0)_{R_{[1]}}}b_R(a\bullet a')_{(1)}\\
&\overset{(\ref{sup2})}{=}&[(a_{(0)}\cdot a'_{(0)_{r_{<-1>}}})
a'_{(0)_{r_{<0>}}}]_{R_{[0]}}\\
&&\ot  
[(a_{(0)}\cdot a'_{(0)_{r_{<-1>}}})
a'_{(0)_{r_{<0>}}}]_{R_{[1]}}b_Ra_{(1)_r}a'_{(1)}\\
&\overset{(\ref{tw4})}{=}&
[(a_{(0)}\cdot a'_{(0)_{r_{<-1>}}})_R
a'_{(0)_{r_{<0>_{{\cal R}}}}}]_{[0]}\\
&& \ot [(a_{(0)}\cdot a'_{(0)_{r_{<-1>}}})_R
a'_{(0)_{r_{<0>_{{\cal R}}}}}]_{[1]}b_{R_{{\cal R}}}a_{(1)_r}a'_{(1)}\\
&\overset{(\ref{sup6})}{=}&
[(a_{(0)}\cdot a'_{(0)_{<-1>}})_R
a'_{(0)_{<0>_{r_{{\cal R}}}}}]_{[0]}\\
&& \ot [(a_{(0)}\cdot a'_{(0)_{<-1>}})_R
a'_{(0)_{<0>_{r_{{\cal R}}}}}]_{[1]}
b_{R_{{\cal R}}}a_{(1)_r}a'_{(1)}\\
&\overset{(\ref{4.7})}{=}&
(a_{(0)}\cdot a'_{(0)_{<-1>}})_{R_{[0]_{(0)}}}
((a_{(0)}\cdot a'_{(0)_{<-1>}})_{R_{[0]_{(1)}}}\cdot 
a'_{(0)_{<0>_{r_{{\cal R}_{\overline{r}_{[0]}}}}}})\\
&& \ot a'_{(0)_{<0>_{r_{{\cal R}_{\overline{r}_{[1]}}}}}}
(a_{(0)}\cdot a'_{(0)_{<-1>}})_{R_{[1]_{\overline{r}}}}
b_{R_{{\cal R}}}a_{(1)_r}a'_{(1)}\\
&\overset{(\ref{extra12})}{=}&
(a_{(0)_{R_{[0]_{(0)}}}}\cdot a'_{(0)_{<-1>}})
(a_{(0)_{R_{[0]_{(1)}}}}\cdot a'_{(0)_{<0>_{r_{{\cal R}_
{\overline{r}_{[0]}}}}}})\\
&&\ot a'_{(0)_{<0>_{r_{{\cal R}_
{\overline{r}_{[1]}}}}}}a_{(0)_{R_{[1]_{\overline{r}}}}}
b_{R_{{\cal R}}}a_{(1)_r}a'_{(1)}, 
\end{eqnarray*}
\begin{eqnarray*}
a_{\tilde{R}}\bullet a'_{\tilde{r}}\ot b_{\tilde{R}_{\tilde{r}}}&=&
a_{(0)_{R_{[0]}}}\bullet a'_{\tilde{r}}\ot 
(a_{(0)_{R_{[1]}}}b_Ra_{(1)})_{\tilde{r}}\\
&=&a_{(0)_{R_{[0]}}}\bullet a'_{(0)_{r_{[0]}}}\ot
a'_{(0)_{r_{[1]}}} 
(a_{(0)_{R_{[1]}}}b_Ra_{(1)})_{r}a'_{(1)}\\
&\overset{(\ref{tw5})}{=}&
a_{(0)_{R_{[0]}}}\bullet a'_{(0)_{r_{{\cal R}_{\overline{r}_{[0]}}}}}\ot
a'_{(0)_{r_{{\cal R}_{\overline{r}_{[1]}}}}}
a_{(0)_{R_{[1]_{\overline{r}}}}}b_{R_{{\cal R}}}a_{(1)_r}a'_{(1)}\\
&=&(a_{(0)_{R_{[0]_{(0)}}}}\cdot a'_{(0)_{r_{{\cal R}_{\overline{r}_{[0]
_{<-1>}}}}}})  
(a_{(0)_{R_{[0]_{(1)}}}}\cdot a'_{(0)_{r_{{\cal R}_{\overline{r}_{[0]
_{<0>}}}}}})\\
&&\ot a'_{(0)_{r_{{\cal R}_{\overline{r}_{[1]}}}}}
a_{(0)_{R_{[1]_{\overline{r}}}}}b_{R_{{\cal R}}}a_{(1)_r}a'_{(1)}\\
&\overset{(\ref{extra11})}{=}&
(a_{(0)_{R_{[0]_{(0)}}}}\cdot a'_{(0)_{r_{{\cal R}_{\overline{r}_{<-1>}}}}})  
(a_{(0)_{R_{[0]_{(1)}}}}\cdot a'_{(0)_{r_{{\cal R}_{\overline{r}_{<0>
_{[0]}}}}}})\\
&&\ot a'_{(0)_{r_{{\cal R}_{\overline{r}_{<0>_{[1]}}}}}}
a_{(0)_{R_{[1]_{\overline{r}}}}}b_{R_{{\cal R}}}a_{(1)_r}a'_{(1)}\\
&\overset{(\ref{sup6})}{=}&
(a_{(0)_{R_{[0]_{(0)}}}}\cdot a'_{(0)_{<-1>}})
(a_{(0)_{R_{[0]_{(1)}}}}\cdot a'_{(0)_{<0>_{r_{{\cal R}_
{\overline{r}_{[0]}}}}}})\\
&&\ot a'_{(0)_{<0>_{r_{{\cal R}_
{\overline{r}_{[1]}}}}}}a_{(0)_{R_{[1]_{\overline{r}}}}}
b_{R_{{\cal R}}}a_{(1)_r}a'_{(1)}, \;\;\;q.e.d. 
\end{eqnarray*} 
\underline{Proof of (\ref{tw4'})}\\
\begin{eqnarray*}
(a\bullet a')_{\tilde{Q}}\ot b_{\tilde{Q}}&=&
(a\bullet a')_{<0>_{Q_{\{0\}}}}\ot (a\bullet a')_{<-1>}b_Q
(a\bullet a')_{<0>_{Q_{\{-1\}}}}\\
&\overset{(\ref{sup3}), (\ref{sup5})}{=}&
[a_{<0>_{(0)_q}}(a_{<0>_{(1)}}\cdot a'_{<0>})]_{Q_{\{0\}}}\ot 
a_{<-1>}a'_{<-1>_q}b_Q\\
&&[a_{<0>_{(0)_q}}(a_{<0>_{(1)}}\cdot a'_{<0>})]_{Q_{\{-1\}}}\\
&\overset{(\ref{tw4'})}{=}&
[a_{<0>_{(0)_{q_{\overline{q}}}}}(a_{<0>_{(1)}}\cdot a'_{<0>})_Q]_{\{0\}}\ot 
a_{<-1>}a'_{<-1>_q}b_{Q_{\overline{q}}}\\
&&[a_{<0>_{(0)_{q_{\overline{q}}}}}(a_{<0>_{(1)}}\cdot a'_{<0>})_Q]_{\{-1\}}\\
&\overset{(\ref{extra3})}{=}&
(a_{<0>_{(0)_{q_{\overline{q}_{\overline{Q}_{\{0\}}}}}}}\cdot 
(a_{<0>_{(1)}}\cdot a'_{<0>})_{Q_{\{0\}_{<-1>}}})
(a_{<0>_{(1)}}\cdot a'_{<0>})_{Q_{\{0\}_{<0>}}}\\ 
&&a_{<-1>}a'_{<-1>_q}b_{Q_{\overline{q}}}
(a_{<0>_{(1)}}\cdot a'_{<0>})_{Q_{\{-1\}_{\overline{Q}}}}
a_{<0>_{(0)_{q_{\overline{q}_{\overline{Q}_{\{-1\}}}}}}}\\
&\overset{(\ref{extra14})}{=}&
(a_{<0>_{(0)_{q_{\overline{q}_{\overline{Q}_{\{0\}}}}}}}\cdot 
a'_{<0>_{Q_{\{0\}_{<-1>}}}})
(a_{<0>_{(1)}}\cdot a'_{<0>_{Q_{\{0\}_{<0>}}}})\\ 
&&a_{<-1>}a'_{<-1>_q}b_{Q_{\overline{q}}}
a'_{<0>_{Q_{\{-1\}_{\overline{Q}}}}}
a_{<0>_{(0)_{q_{\overline{q}_{\overline{Q}_{\{-1\}}}}}}}, 
\end{eqnarray*}
\begin{eqnarray*}
a_{\tilde{q}}\bullet a'_{\tilde{Q}}\ot b_{\tilde{Q}_{\tilde{q}}}&=&
a_{\tilde{q}}\bullet a'_{<0>_{Q_{\{0\}}}}\ot 
(a'_{<-1>}b_Qa'_{<0>_{Q_{\{-1\}}}})_{\tilde{q}}\\
&=& a_{<0>_{q_{\{0\}}}}\bullet  a'_{<0>_{Q_{\{0\}}}}\ot a_{<-1>}
(a'_{<-1>}b_Qa'_{<0>_{Q_{\{-1\}}}})_qa_{<0>_{q_{\{-1\}}}}\\
&\overset{(\ref{tw5'})}{=}&
a_{<0>_{q_{\overline{q}_{\overline{Q}_{\{0\}}}}}}\bullet  
a'_{<0>_{Q_{\{0\}}}}\ot a_{<-1>}
a'_{<-1>_q}b_{Q_{\overline{q}}}a'_{<0>_{Q_{\{-1\}_{\overline{Q}}}}}
a_{<0>_{q_{\overline{q}_{\overline{Q}_{\{-1\}}}}}}\\
&=&(a_{<0>_{q_{\overline{q}_{\overline{Q}_{\{0\}_{(0)}}}}}}\cdot   
a'_{<0>_{Q_{\{0\}_{<-1>}}}})
(a_{<0>_{q_{\overline{q}_{\overline{Q}_{\{0\}_{(1)}}}}}}\cdot   
a'_{<0>_{Q_{\{0\}_{<0>}}}})\\
&&\ot a_{<-1>}
a'_{<-1>_q}b_{Q_{\overline{q}}}a'_{<0>_{Q_{\{-1\}_{\overline{Q}}}}}
a_{<0>_{q_{\overline{q}_{\overline{Q}_{\{-1\}}}}}}\\
&\overset{(\ref{extra13})}{=}&
(a_{<0>_{q_{\overline{q}_{\overline{Q}_{(0)_{\{0\}}}}}}}\cdot   
a'_{<0>_{Q_{\{0\}_{<-1>}}}})
(a_{<0>_{q_{\overline{q}_{\overline{Q}_{(1)}}}}}\cdot   
a'_{<0>_{Q_{\{0\}_{<0>}}}})\\
&&\ot a_{<-1>}
a'_{<-1>_q}b_{Q_{\overline{q}}}a'_{<0>_{Q_{\{-1\}_{\overline{Q}}}}}
a_{<0>_{q_{\overline{q}_{\overline{Q}_{(0)_{\{-1\}}}}}}}\\
&\overset{(\ref{sup5})}{=}&
(a_{<0>_{(0)_{q_{\overline{q}_{\overline{Q}_{\{0\}}}}}}}\cdot 
a'_{<0>_{Q_{\{0\}_{<-1>}}}})
(a_{<0>_{(1)}}\cdot a'_{<0>_{Q_{\{0\}_{<0>}}}})\\ 
&&a_{<-1>}a'_{<-1>_q}b_{Q_{\overline{q}}}
a'_{<0>_{Q_{\{-1\}_{\overline{Q}}}}}
a_{<0>_{(0)_{q_{\overline{q}_{\overline{Q}_{\{-1\}}}}}}}, \;\;\;q.e.d.
\end{eqnarray*}
\underline{Proof of (\ref{comb1})}\\
\begin{eqnarray*}
b_{\tilde{R}}\ot a_{\tilde{R}_{\tilde{Q}}}\ot b'_{\tilde{Q}}&=&
a_{(0)_{R_{[1]}}}b_Ra_{(1)}\ot (a_{(0)_{R_{[0]}}})_{\tilde{Q}}\ot 
b'_{\tilde{Q}}\\
&=&a_{(0)_{R_{[1]}}}b_Ra_{(1)}\ot a_{(0)_{R_{[0]_{<0>_{Q_{\{0\}}}}}}}\ot  
a_{(0)_{R_{[0]_{<-1>}}}}
b'_Qa_{(0)_{R_{[0]_{<0>_{Q_{\{-1\}}}}}}}\\
&\overset{(\ref{extra10})}{=}&
a_{(0)_{<0>_{R_{Q_{[1]}}}}}b_Ra_{(1)}\ot 
a_{(0)_{<0>_{R_{Q_{[0]_{\{0\}}}}}}}\ot  
a_{(0)_{<-1>}}b'_Qa_{(0)_{<0>_{R_{Q_{[0]_{\{-1\}}}}}}}\\
&\overset{(\ref{extra6})}{=}&
a_{(0)_{<0>_{R_{Q_{\{0\}_{[1]}}}}}}b_Ra_{(1)}\ot 
a_{(0)_{<0>_{R_{Q_{\{0\}_{[0]}}}}}}\ot  
a_{(0)_{<-1>}}b'_Qa_{(0)_{<0>_{R_{Q_{\{-1\}}}}}},\\
\end{eqnarray*}
\begin{eqnarray*}
b_{\tilde{R}}\ot a_{\tilde{Q}_{\tilde{R}}}\ot b'_{\tilde{Q}}&=&
b_{\tilde{R}}\ot (a_{<0>_{Q_{\{0\}}}})_{\tilde{R}}\ot a_{<-1>}b'_Q
a_{<0>_{Q_{\{-1\}}}}\\
&=&a_{<0>_{Q_{\{0\}_{(0)_{R_{[1]}}}}}}b_Ra_{<0>_{Q_{\{0\}_{(1)}}}}\ot 
a_{<0>_{Q_{\{0\}_{(0)_{R_{[0]}}}}}}\ot a_{<-1>}b'_Q
a_{<0>_{Q_{\{-1\}}}}\\
&\overset{(\ref{extra9})}{=}&
a_{(0)_{R_{<0>_{Q_{\{0\}_{[1]}}}}}}b_Ra_{(1)}\ot 
a_{(0)_{R_{<0>_{Q_{\{0\}_{[0]}}}}}}\ot a_{(0)_{R_{<-1>}}}b'_Q
a_{(0)_{R_{<0>_{Q_{\{-1\}}}}}}\\
&\overset{(\ref{sup6})}{=}&
a_{(0)_{<0>_{R_{Q_{\{0\}_{[1]}}}}}}b_Ra_{(1)}\ot 
a_{(0)_{<0>_{R_{Q_{\{0\}_{[0]}}}}}}\ot  
a_{(0)_{<-1>}}b'_Qa_{(0)_{<0>_{R_{Q_{\{-1\}}}}}},\;\;\;q.e.d.
\end{eqnarray*}
\underline{Proof of (\ref{comb2})}\\
\begin{eqnarray*}
a_{\tilde{R}}\ot b_{\tilde{R}_{\tilde{Q}}}\ot a'_{\tilde{Q}}&=&
a_{(0)_{r_{[0]}}}\ot (a_{(0)_{r_{[1]}}}b_ra_{(1)})_{\tilde{Q}}\ot 
a'_{\tilde{Q}}\\
&=&
a_{(0)_{r_{[0]}}}\ot a'_{<-1>}(a_{(0)_{r_{[1]}}}b_ra_{(1)})_{Q}
a'_{<0>_{Q_{\{-1\}}}}
\ot a'_{<0>_{Q_{\{0\}}}}\\
&\overset{(\ref{tw5'})}{=}&
a_{(0)_{r_{[0]}}}\ot a'_{<-1>}a_{(0)_{r_{[1]_q}}}b_{r_Q}a_{(1)_
{\overline{q}}}
a'_{<0>_{q_{Q_{\overline{q}_{\{-1\}}}}}}
\ot a'_{<0>_{q_{Q_{\overline{q}_{\{0\}}}}}}\\
&\overset{(\ref{extra8})}{=}&
a_{(0)_{r_{R_{[0]}}}}\ot 
a_{(0)_{r_{R_{[1]}}}}a'_{<-1>_R}b_{r_Q}a_{(1)_q}
a'_{<0>_{Q_{q_{\{-1\}}}}}
\ot a'_{<0>_{Q_{q_{\{0\}}}}},
\end{eqnarray*}
\begin{eqnarray*}
a_{\tilde{R}}\ot b_{\tilde{Q}_{\tilde{R}}}\ot a'_{\tilde{Q}}&=&
a_{\tilde{R}}\ot (a'_{<-1>}b_Qa'_{<0>_{Q_{\{-1\}}}})_{\tilde{R}}\ot 
a'_{<0>_{Q_{\{0\}}}}\\
&=&
a_{(0)_{R_{[0]}}}\ot a_{(0)_{R_{[1]}}}
(a'_{<-1>}b_Qa'_{<0>_{Q_{\{-1\}}}})_Ra_{(1)}\ot 
a'_{<0>_{Q_{\{0\}}}}\\
&\overset{(\ref{tw5})}{=}&
a_{(0)_{R_{r_{{\cal R}_{[0]}}}}}\ot a_{(0)_{R_{r_{{\cal R}_{[1]}}}}}
a'_{<-1>_{{\cal R}}}b_{Q_r}a'_{<0>_{Q_{\{-1\}_R}}}a_{(1)}\ot 
a'_{<0>_{Q_{\{0\}}}}\\
&\overset{(\ref{comb2})}{=}&
a_{(0)_{R_{r_{{\cal R}_{[0]}}}}}\ot a_{(0)_{R_{r_{{\cal R}_{[1]}}}}}
a'_{<-1>_{{\cal R}}}b_{r_Q}a'_{<0>_{Q_{\{-1\}_R}}}a_{(1)}\ot 
a'_{<0>_{Q_{\{0\}}}}\\
&\overset{(\ref{extra7})}{=}&
a_{(0)_{r_{R_{[0]}}}}\ot 
a_{(0)_{r_{R_{[1]}}}}a'_{<-1>_R}b_{r_Q}a_{(1)_q}
a'_{<0>_{Q_{q_{\{-1\}}}}}
\ot a'_{<0>_{Q_{q_{\{0\}}}}},\;\;\;q.e.d.
\end{eqnarray*}

We prove now that the map 
$\varphi :\tilde{A}\; _{\tilde{Q}}\ot _{\tilde{R}}B\rightarrow A\; _Q\ot _RB$,    
$\varphi (a\ot b)=a_{(0)_{<0>}}\ot a_{(1)}ba_{(0)_{<-1>}}$, is an algebra 
isomorphism. First, using (\ref{sup4}), (\ref{4.8}), (\ref{extra1}),  
(\ref{extra6}), (\ref{4.9}), (\ref{extra2}), it is easy to see that 
$\varphi $ is bijective, with inverse given by $a\ot b\mapsto 
a_{[0]_{\{0\}}}\ot a_{[1]}ba_{[0]_{\{-1\}}}$. It is obvious that 
$\varphi (1\ot 1)=1\ot 1$, so we only have to prove that $\varphi $ is 
multiplicative. We compute:
\begin{eqnarray*}
\varphi ((a\ot b)(a'\ot b'))&=&\varphi (a_{\tilde{Q}}\bullet 
a'_{\tilde{R}}\ot b_{\tilde{R}}b'_{\tilde{Q}})\\
&=&\varphi (a_{<0>_{Q_{\{0\}}}}\bullet a'_{(0)_{R_{[0]}}}\ot 
a'_{(0)_{R_{[1]}}}b_Ra'_{(1)}a_{<-1>}b'_Qa_{<0>_{Q_{\{-1\}}}})\\
&=&[a_{<0>_{Q_{\{0\}}}}\bullet a'_{(0)_{R_{[0]}}}]_{(0)_{<0>}}\ot 
[a_{<0>_{Q_{\{0\}}}}\bullet a'_{(0)_{R_{[0]}}}]_{(1)}a'_{(0)_{R_{[1]}}}b_R
a'_{(1)}\\
&&a_{<-1>}b'_Qa_{<0>_{Q_{\{-1\}}}}
[a_{<0>_{Q_{\{0\}}}}\bullet a'_{(0)_{R_{[0]}}}]_{(0)_{<-1>}}\\
&\overset{(\ref{sup2})}{=}&[(a_{<0>_{Q_{\{0\}_{(0)}}}}\cdot 
a'_{(0)_{R_{[0]_{(0)_{r_{<-1>}}}}}})a'_{(0)_{R_{[0]_{(0)_{r_{<0>}}}}}}]
_{<0>}\\
&&\ot a_{<0>_{Q_{\{0\}_{(1)_r}}}}a'_{(0)_{R_{[0]_{(1)}}}}
a'_{(0)_{R_{[1]}}}b_Ra'_{(1)}a_{<-1>}b'_Qa_{<0>_{Q_{\{-1\}}}}\\
&&[(a_{<0>_{Q_{\{0\}_{(0)}}}}\cdot
a'_{(0)_{R_{[0]_{(0)_{r_{<-1>}}}}}})a'_{(0)_{R_{[0]_{(0)_{r_{<0>}}}}}}]
_{<-1>}\\
&\overset{(\ref{4.9})}{=}&[(a_{<0>_{Q_{\{0\}_{(0)}}}}\cdot 
a'_{(0)_{R_{r_{<-1>}}}})a'_{(0)_{R_{r_{<0>}}}}]_{<0>}\\
&&\ot  a_{<0>_{Q_{\{0\}_{(1)_r}}}}b_Ra'_{(1)}a_{<-1>}b'_Q
a_{<0>_{Q_{\{-1\}}}}\\ 
&&[(a_{<0>_{Q_{\{0\}_{(0)}}}}\cdot 
a'_{(0)_{R_{r_{<-1>}}}})a'_{(0)_{R_{r_{<0>}}}}]_{<-1>}\\
&\overset{(\ref{extra4})}{=}&a_{<0>_{Q_{\{0\}_{(0)_{<0>_q}}}}}
a'_{(0)_{R_{r_{<0>}}}}\ot a_{<0>_{Q_{\{0\}_{(1)_r}}}}
b_Ra'_{(1)}a_{<-1>}b'_Q\\
&&a_{<0>_{Q_{\{-1\}}}}a_{<0>_{Q_{\{0\}_{(0)_{<-1>}}}}}
a'_{(0)_{R_{r_{<-1>_q}}}}\\
&\overset{(\ref{sup4})}{=}&a_{<0>_{Q_{\{0\}_{<0>_{(0)_q}}}}}
a'_{(0)_{R_{r_{<0>}}}}\ot a_{<0>_{Q_{\{0\}_{<0>_{(1)_r}}}}}
b_Ra'_{(1)}a_{<-1>}b'_Q\\
&&a_{<0>_{Q_{\{-1\}}}}a_{<0>_{Q_{\{0\}_{<-1>}}}}
a'_{(0)_{R_{r_{<-1>_q}}}}\\
&\overset{(\ref{extra2})}{=}&a_{<0>_{Q_{(0)_q}}}
a'_{(0)_{R_{r_{<0>}}}}\ot a_{<0>_{Q_{(1)_r}}}
b_Ra'_{(1)}a_{<-1>}b'_Qa'_{(0)_{R_{r_{<-1>_q}}}}\\
&\overset{(\ref{sup5})}{=}&a_{<0>_{(0)_{Q_q}}}
a'_{(0)_{R_{r_{<0>}}}}\ot a_{<0>_{(1)_r}}
b_Ra'_{(1)}a_{<-1>}b'_Qa'_{(0)_{R_{r_{<-1>_q}}}}\\
&\overset{(\ref{sup4})}{=}&a_{(0)_{<0>_{Q_q}}}
a'_{(0)_{R_{r_{<0>}}}}\ot a_{(1)_r}
b_Ra'_{(1)}a_{(0)_{<-1>}}b'_Qa'_{(0)_{R_{r_{<-1>_q}}}}\\
&\overset{(\ref{sup6})}{=}&a_{(0)_{<0>_{Q_q}}}
a'_{(0)_{<0>_{R_r}}}\ot a_{(1)_r}
b_Ra'_{(1)}a_{(0)_{<-1>}}b'_Qa'_{(0)_{<-1>_q}}, 
\end{eqnarray*}  
\begin{eqnarray*}
\varphi (a\ot b)\varphi (a'\ot b')&=&
(a_{(0)_{<0>}}\ot a_{(1)}ba_{(0)_{<-1>}})
(a'_{(0)_{<0>}}\ot a'_{(1)}b'a'_{(0)_{<-1>}})\\
&=&a_{(0)_{<0>_Q}}a'_{(0)_{<0>_R}}\ot (a_{(1)}ba_{(0)_{<-1>}})_R
(a'_{(1)}b'a'_{(0)_{<-1>}})_Q\\
&\overset{(\ref{tw5}), (\ref{tw5'})}{=}&
a_{(0)_{<0>_{Q_{q_{\overline{q}}}}}}a'_{(0)_{<0>_{R_{r_{\overline{r}}}}}}
\ot a_{(1)_{\overline{r}}}b_ra_{(0)_{<-1>_R}}a'_{(1)_Q}b'_q
a'_{(0)_{<-1>_{\overline{q}}}}\\
&\overset{(\ref{extra5})}{=}&
a_{(0)_{<0>_{Q_q}}}a'_{(0)_{<0>_{R_r}}}
\ot a_{(1)_r}b_Ra'_{(1)}a_{(0)_{<-1>}}b'_Q
a'_{(0)_{<-1>_q}},
\end{eqnarray*}
finishing the proof.
\end{proof}
\begin{remark}
It is very easy to see that Theorem \ref{invundtw} generalizes 
Theorem 4.4 in \cite{jlpvo}. On the other hand, it generalizes also 
Proposition \ref{invutwlrsmash}. Indeed, Proposition \ref{invutwlrsmash} 
may be obtained by taking in Theorem \ref{invundtw} $A={\cal A}$, 
$B=H$, $R:H\otimes {\cal A}\rightarrow {\cal A}\otimes H$, 
$R(h\otimes \varphi )=h_1\cdot \varphi \otimes h_2$, $Q:{\cal A}\otimes H 
\rightarrow {\cal A}\otimes H$, $Q(\varphi \otimes h)=
\varphi \cdot h_2\otimes h_1$, 
$\mu _l:H\otimes {\cal A}\rightarrow {\cal A}$, $\mu _l(h\otimes 
\varphi )=h\cdot \varphi $,  
$\mu _r:{\cal A}\otimes H\rightarrow {\cal A}$, $\mu _r(\varphi 
\otimes h)=\varphi \cdot h$,  
$\rho _r:{\cal A}\rightarrow {\cal A}\otimes H$, 
$\rho _r(\varphi )=G^1\cdot \varphi \otimes G^2$, 
$\rho _l:{\cal A}\rightarrow H\otimes {\cal A}$, 
$\rho _l(\varphi )=F^1\otimes \varphi \cdot F^2$, 
$\lambda _r:{\cal A}\rightarrow {\cal A}\otimes H$, 
$\lambda _r(\varphi )=F^1\cdot \varphi \otimes F^2$, 
$\lambda _l:{\cal A}\rightarrow H\otimes {\cal A}$, 
$\lambda _l(\varphi )=G^1\otimes \varphi \cdot G^2$,  
for all $h\in H$ and $\varphi \in {\cal A}$. 
\end{remark}

\end{document}